\documentclass[12pt,leqno]{amsart}
\usepackage{amssymb,amsmath,graphicx,amsfonts,euscript}

\newtheorem{thm}{Theorem}[section]

\newtheorem{cor}[thm]{Corollary}

\newtheorem{define}[thm]{Definition}

\newtheorem{lemma}[thm]{Lemma}
\newtheorem{assume}[thm]{Condition}

\newcommand{\n}{\nonumber}

\renewcommand{\th}{\theta}

\renewcommand{\a}{\alpha}

\newcommand{\nao}{\nabla^\bot}
\newcommand{\na}{\nabla^\bot \theta}
\renewcommand{\k}{\kappa}

\newcommand{\xx}{\xi(x,t)}
\newcommand{\xxy}{\xi (x+y,t)}
\newcommand{\yh}{\hat{y}}
\newcommand{\fn}{\|_{\mathcal{\dot{F}}^{\sigma} _{p_1 , q}}}
\newcommand{\bb}{\begin{equation}}
\newcommand{\ee}{\end{equation}}
\newcommand{\bq}{\begin{eqnarray}}
\newcommand{\eq}{\end{eqnarray}}
\newcommand{\bqn}{\begin{eqnarray*}}
\newcommand{\eqn}{\end{eqnarray*}}
\renewcommand{\Bbb}{\mathbb}
\def\pp{\partial}

\numberwithin{equation}{section}

\begin{document}
\title[Generalized surface quasi-geostrophic equations]
{Dissipative models generalizing the 2D Navier-Stokes and the surface quasi-geostrophic equations}
\author[Dongho Chae, Peter Constantin and Jiahong Wu]{Dongho Chae$^{1}$, Peter Constantin$^{2}$ and Jiahong Wu$^{3}$}
\address{$^1$ Department of Mathematics,
Sungkyunkwan University,
Suwon 440-746, Korea}

\address{$^2$Department of Mathematics,
 University of Chicago,
 5734 S. University Avenue,
Chicago, IL 60637, USA.}

\address{$^3$Department of Mathematics,
Oklahoma State University,
401 Mathematical Sciences,
Stillwater, OK 74078, USA.}

\email{chae@skku.edu}
\email{const@cs.uchicago.edu}
\email{jiahong@math.okstate.edu}

\begin{abstract}
This paper is devoted to the global (in time) regularity problem for a family
of active scalar equations with fractional dissipation. Each component of the velocity field $u$ is determined by the active scalar $\theta$ through $\mathcal{R} \Lambda^{-1} P(\Lambda) \theta$ where $\mathcal{R}$ denotes a Riesz transform, $\Lambda=(-\Delta)^{1/2}$ and $P(\Lambda)$ represents a family of Fourier multiplier operators. The 2D Navier-Stokes vorticity equations correspond to the special case $P(\Lambda)=I$ while the surface quasi-geostrophic (SQG) equation to $P(\Lambda) =\Lambda$. We obtain the global regularity for a class of equations for which $P(\Lambda)$ and the fractional power of the dissipative Laplacian are required to satisfy an explicit condition. In particular, the active scalar equations with any fractional dissipation and with $P(\Lambda)= (\log(I-\Delta))^\gamma$ for any $\gamma>0$ are globally regular.
\end{abstract}

\maketitle

\section{Introduction}
\label{intr}
\setcounter{equation}{0}

This paper is devoted to the dissipative active scalar equation
\begin{equation} \label{general}
\left\{
\begin{array}{l}
\pp_t \theta + u\cdot\nabla \theta + \kappa (-\Delta)^\alpha \theta=0, \quad x\in \mathbb{R}^d, \, t>0, \\
u = (u_j), \quad u_j = \mathcal{R}_l \Lambda^{-1} P(\Lambda)\, \theta,\quad 1\le j, \,l\le d,
\end{array}
\right.
\end{equation}
where $\kappa>0$ and $\alpha>0$ are parameters, $\theta =\theta(x,t)$ is a scalar function of $x\in \mathbb{R}^{d}$ and $t\ge 0$, $u$ denotes a velocity field with each of its components $u_j$ ($1\le j\le d$) given by a Riesz transform $\mathcal{R}_l$ applied to $\Lambda^{-1} P(\Lambda)\, \theta$.
Here the operators $\Lambda = (-\Delta)^{\frac12}$, $P(\Lambda)$ and $\mathcal{R}_l$ are defined through their Fourier transforms,
$$
\widehat{\Lambda f}(\xi) = |\xi| \widehat{f}(\xi), \quad \widehat{P(\Lambda) f}(\xi) = P(|\xi|) \widehat{f}(\xi), \quad \widehat{\mathcal{R}_l f}(\xi)= \frac{i\,\xi_l}{|\xi|}\, \widehat{f}(\xi),
$$
where $1\le l\le d$ is an integer, $\widehat{f}$ or $\mathcal{F}(f)$ denotes the Fourier transform,
$$
\widehat{f}(\xi) = \mathcal{F}(f)(\xi) =\frac{1}{(2\pi)^{d/2}} \int_{\mathbb{R}^d} e^{-i x\cdot \xi} f(x)\,dx.
$$
We are primarily concerned with the global (in time) regularity issue concerning solutions of (\ref{general}) with a given initial data
\begin{equation} \label{IC}
\theta(x,0) =\theta_0(x), \quad x\in \mathbb{R}^d.
\end{equation}

\vskip .1in
A special example of (\ref{general}) is the 2D active scalar equation
\begin{equation} \label{general2d}
\left\{
\begin{array}{l}
\pp_t \theta + u\cdot\nabla \theta + \kappa (-\Delta)^\alpha \theta=0, \quad x\in \mathbb{R}^2, \, t>0, \\
u = \nabla^\perp \psi\equiv (-\pp_{x_2}\psi, \pp_{x_1} \psi), \quad
\Delta \psi = P(\Lambda)\, \theta
\end{array}
\right.
\end{equation}
which includes as special cases the 2D Navier-Stokes vorticity equation
\begin{equation}\label{euler}
\left\{
\begin{array}{l}
\pp_t \omega + u \cdot \nabla \omega-\nu \Delta \omega =0,\\
u =\nabla^\perp \psi, \quad \Delta\psi=\omega
\end{array}
\right.
\end{equation}
and the dissipative surface quasi-geostrophic (SQG) equation
\begin{equation}\label{SQG}
\left\{
\begin{array}{l}
\pp_t \theta + u \cdot \nabla \theta + \kappa (-\Delta)^\alpha \theta= 0,\\u=\nabla^\perp \psi, \quad -\Lambda\psi = \theta.
\end{array}
\right.
\end{equation}
There are numerous studies on the Navier-Stokes equations and the global regularity in the 2D case has long been established (see e.g. \cite{ConF}, \cite{DoGi} and \cite{MaBe}). The SQG equation models the dynamics of the potential temperature $\theta$ of the 3D quasi-geostrophic equations on the 2D horizontal boundaries and is useful in modeling atmospheric phenomena such as the frontogenesis (see e.g. \cite{CMT}, \cite{MaTa} and \cite{Pe}). The SQG equation (inviscid or dissipative) is
also mathematically important. As detailed in \cite{CMT}, the behavior of its strongly nonlinear solutions are strikingly analogous to that of the potentially singular solutions of the 3D incompressible Navier-Stokes
and the Euler equations. The global regularity issue concerning the SQG equation has recently been studied very extensively and many important progress has been made (see e.g. \cite{AbHm}, \cite{Bae}, \cite{Bar}, \cite{Blu}, \cite{CaS}, \cite{CV}, \cite{CaFe}, \cite{Ch}, \cite{ChJDE}, \cite{Cha}, \cite{Cha2}, \cite{Cha4}, \cite{CCCF}, \cite{ChL}, \cite{Cham}, \cite{CMZ1}, \cite{Chen}, \cite{Con}, \cite{CCW}, \cite{CIW}, \cite{CLS}, \cite{CMT}, \cite{CNS}, \cite{CWnew1}, \cite{CWnew2}, \cite{Cor}, \cite{CC}, \cite{CoFe1}, \cite{CoFe2}, \cite{CoFe3}, \cite{CFMR}, \cite{Dab}, \cite{DHLY}, \cite{DoCh}, \cite{Dong}, \cite{DoDu}, \cite{DoLi0}, \cite{DoLi}, \cite{DoPo}, \cite{DoPo2}, \cite{FPV}, \cite{FrVi}, \cite{Gil}, \cite{HPGS}, \cite{HmKe}, \cite{HmKe2}, \cite{Ju}, \cite{Ju2}, \cite{KhTi}, \cite{Ki1}, \cite{Ki2}, \cite{Kinew1}, \cite{KN1}, \cite{KN2}, \cite{KNV}, \cite{Li}, \cite{LiRo}, \cite{Maj}, \cite{MaBe}, \cite{MaTa}, \cite{Mar1}, \cite{Mar2}, \cite{Mar3}, \cite{MarLr}, \cite{May}, \cite{MayZ}, \cite{MiXu}, \cite{Mi}, \cite{NiSc}, \cite{OhYa1}, \cite{Pe}, \cite{ReDr}, \cite{Res}, \cite{Ro1}, \cite{Ro2}, \cite{Sch}, \cite{Sch2}, \cite{Si}, \cite{Si2}, \cite{Sta}, \cite{WaJi}, \cite{WaZh}, \cite{Wu97}, \cite{Wu2}, \cite{Wu01}, \cite{Wu02}, \cite{Wu3}, \cite{Wu4}, \cite{Wu41}, \cite{Wu31}, \cite{Wu77}, \cite{Yu}, \cite{Yuan}, \cite{YuanJ}, \cite{Zha0}, \cite{Zha}, \cite{Zhou}, \cite{Zhou2}). In particular, the global regularity for the critical case $\alpha=1/2$
has been successfully established (\cite{CV}, \cite{KNV}). The situation in the supercritical case $\alpha<1/2$ is only partially understood at the time of writing. The results in \cite{CWnew1}, \cite{CWnew2} and \cite{DoPo} imply that any solution of the supercritical SQG equation can develop potential finite time singularity only in the regularity window between $L^\infty$ and $C^\delta$ with $\delta<1-2\alpha$. Several very recent preprints on the supercritical case also revealed some very interesting properties of the supercritical dissipation (\cite{Bar}, \cite{Dab}, \cite{Kinew1}, \cite{Si}).

\vskip .1in
Our goal here is to establish the global regularity of (\ref{general}) for more general operators $P$. In particular, we are interested in the global regularity of the intermediate equations between the 2D Navier-Stokes equation and the supercritical SQG equation. This paper is a continuation of our previous study on the inviscid counterpart of (\ref{general}) (\cite{ChCW}).  The consideration here is restricted to $P$ satisfying the following condition.
\begin{assume} \label{P_con}
The symbol $P=P(|\xi|)$ assumes the following properties:
\begin{enumerate}
\item $P$ is continuous on $\mathbb{R}^d$ and $P\in C^\infty(\mathbb{R}^d\setminus\{0\})$;
\item $P$ is radially symmetric;
\item $P=P(|\xi|)$ is nondecreasing in $|\xi|$;
\item There exist two constants $C$ and $C_0$ such that
\begin{equation*}
\sup_{2^{-1} \le |\eta| \le 2}\, \left|(I-\Delta_\eta)^n \,P(2^j |\eta|)\right| \le C\, P(C_0 \, 2^j)
\end{equation*}
for any integer $j$ and $n=1,2,\cdots, 1+ \left[\frac{d}{2}\right]$.
\end{enumerate}
\end{assume}
We remark that (4) in Condition \ref{P_con} is a very natural condition on symbols of Fourier multiplier operators and is similar to the main condition in the Mihlin-H\"{o}rmander Multiplier Theorem (see e.g. \cite[p.96]{St}). For notational convenience, we also assume that $P\ge 0$. Some special examples of $P$ are
\begin{eqnarray*}
&& P(\xi) = \left(\log(1 +|\xi|^2)\right)^\gamma \quad\mbox{with $\gamma\ge 0$}, \\
&& P(\xi) = \left(\log(1+\log(1 +|\xi|^2))\right)^\gamma \quad\mbox{with $\gamma\ge 0$}, \\
&& P(\xi) = |\xi|^\beta \quad\mbox{with $\beta\ge 0$},\\
&& P(\xi) = (\log(1 +|\xi|^2))^\gamma\,|\xi|^\beta \quad\mbox{with $\gamma\ge 0$ and $\beta\ge 0$}.
\end{eqnarray*}

\vskip .1in
As in the study of the Navier-Stokes and the Euler equations, the quantity $\|\nabla u\|_{L^\infty}$ plays a crucial role in the global regularity issue. In our previous work on the inviscid counterpart of (\ref{general}), we established bounds for the building blocks $\|\nabla \Delta_j u\|_{L^q}$ and $\|\nabla S_N u\|_{L^q}$ for $1\le q\le \infty$. More precisely, the following theorem is proven in \cite{ChCW}.

\begin{thm} \label{nablau}
Let $u: \mathbb{R}^d\to \mathbb{R}^d$ be a vector field. Assume that $u$ is related to a scalar $\theta$ by
$$
(\nabla u)_{jk}  = \mathcal{R}_l \mathcal{R}_m\, P(\Lambda) \, \theta,
$$
where $1\le j,k,l,m\le d$, $(\nabla u)_{jk}$ denotes the $(j,k)$-th entry of $\nabla u$, $\mathcal{R}_l$ denotes the Riesz transform, and $P$ obeys Condition \ref{P_con}. Then, for any integers $j\ge 0$ and $N\ge 0$,
\begin{eqnarray}
\|S_N \nabla u\|_{L^p} &\le& C_{p,d}\, P(C_0 2^N)\,\|S_{N} \theta\|_{L^p}, \quad 1<p<\infty,
\label{bound1} \\
\|\Delta_j \nabla u\|_{L^q} &\le& C_d\,  P(C_0 2^j)\,\|\Delta_j \theta\|_{L^q}, \quad 1\le q\le \infty,
\label{bound2} \\
\|S_N \nabla u\|_{L^\infty} &\le&  C_d\,\|\theta\|_{L^1\cap L^\infty} + C_d\, N\,P(C_0 2^N)\,\|S_{N+1}\theta\|_{L^\infty}, \label{bound3}
\end{eqnarray}
where $C_{p,d}$ is a constant depending on $p$ and $d$ only and $C_d$s' depend on $d$ only.
\end{thm}

With the aid of these bounds, we were able to show in \cite{ChCW}
that (\ref{general}) with $\kappa=0$ and $P(\Lambda)=\left(\log(1+\log(1 -\Delta))\right)^\gamma$ for
$0\le \gamma\le 1$ has a unique global (in time) solution in the Besov space
$B^s_{q,\infty}(\mathbb{R}^d)$ with $d<q\le \infty$ and $s>1$. In addition, a
regularity criterion is also provided in \cite{ChCW} for (\ref{general}) with
$P(\Lambda) = \Lambda^\beta$ for $0\le \beta\le 1$. Our goal here is to extend our
study to cover more general operators when we turn on the dissipation. Indeed we
are able to establish the global existence and uniqueness for a very general family
of symbols. Before stating the result, we introduce the extended Besov spaces. Here $\mathcal{S}'$ denotes the class of tempered distributions and $\Delta_j$ with $j\ge -1$ denotes the standard Fourier localization operator. The notation $\Delta_j$, $S_N$ and Besov spaces are now quite standard and can be found in several books and many papers (see e.g. \cite{BL}, \cite{Che}, \cite{RuSi}, \cite{Tr}). They can also be found in Appendix A of \cite{ChCW}.

\begin{define}
Let $s\in \mathbb{R}$ and $1\le q,r\le \infty$. Let $A=\{A_j\}_{j\ge -1}$ with $A_j\ge 0$ be a nondecreasing sequence. The extended Besov space $B^{s,A}_{q,r}$ consists of $f\in \mathcal{S}'(\mathbb{R}^d)$ satisfying
$$
\|f\|_{B^{s,A}_{q,r}} \equiv  \left\|2^{s A_j}\, \|\Delta_j f\|_{L^q(\mathbb{R}^d)} \right\|_{l^r} \,< \infty.
$$
\end{define}

Obviously, when $A_j = j+1$, $B^{s,A}_{q,r}$ becomes the standard inhomogeneous Besov space $B^s_{q,r}$. When $A_j =o(j+1)$ as $j\to \infty$, $B^{s,A}_{q,r}$ is a less regular class than the corresponding Besov space $B^s_{q,r}$; we will refer to these spaces as sub-Besov spaces. When $j=o(A_j)$, $B^{s,A}_{q,r}$, we will refer to the spaces as super-Besov spaces.

\vskip .1in
With these definitions at our disposal, our main theorem can be stated as follows.
\begin{thm} \label{main1}
Consider the dissipative active scalar equation (\ref{general}) with $\kappa>0$, $\alpha>0$ and $P(\xi)$ satisfying Condition \ref{P_con}. Let $s>1$, $2\le q\le \infty$ and $A=\{A_j\}_{j\ge -1}$ be a nondecreasing sequence with $A_j\ge 0$. Let $\theta_0 \in L^1(\mathbb{R}^d)\cap L^\infty(\mathbb{R}^d) \cap B^{s,A}_{q,\infty}(\mathbb{R}^d)$. Assume either the velocity $u$ is divergence-free or the solution $\theta$ is bounded in $L^1(\mathbb{R}^d)\cap L^\infty(\mathbb{R}^d)$ for all time. If, there exists a constant $C$ such that for all $j\ge -1$,
\begin{equation}\label{sb}
\sum_{k\ge j-1, k\ge -1} \frac{2^{s A_{j-2}}\,P(2^{k+1})}{2^{s A_k}\,P(2^{j+1})} < C
\end{equation}
and
\begin{equation}\label{decay}
\kappa^{-1}\, 2^{s(A_j-A_{j-2})}\, (j+2) P(2^{j+2})\,2^{-2\alpha j} \to 0\quad \mbox{as} \quad j\to \infty,
\end{equation}
then (\ref{general}) has a unique global solution $\theta$ satisfying
$$
\theta \in L^\infty\left([0,\infty); B^{s,A}_{q,\infty}(\mathbb{R}^d)\right).
$$
\end{thm}

We single out two special consequences of Theorem \ref{main1}. In the case when
\begin{equation}\label{PA}
P(|\xi|) = \left(\log(I+|\xi|^2)\right)^\gamma,\,\, \gamma\ge 0\quad\mbox{and}\quad A_j=(j+1)^b\quad\mbox{for some $b\le 1$},
\end{equation}
(\ref{sb}) is trivially satisfied and the condition in (\ref{decay}) reduces to
\begin{equation}\label{ed}
2^{s((j+1)^b - j^b)}\, (j+2)^{1+\gamma} 2^{-2\alpha j} \to 0 \quad \mbox{as} \quad j\to \infty,
\end{equation}
which is obviously satisfied for any $\alpha>0$. We thus obtain the following corollary.
\begin{cor}
Consider the dissipative Log-Euler equation
\begin{equation} \label{Log-Euler}
\left\{
\begin{array}{l}
\pp_t \theta + u\cdot\nabla \theta + \kappa (-\Delta)^\alpha\theta =0, \\
u = \nabla^\perp \psi, \quad
\Delta \psi = \left(\log(1 -\Delta)\right)^\gamma\, \theta
\end{array}
\right.
\end{equation}
with $\kappa>0$, $\alpha>0$ and $\gamma\ge 0$. Assume that $\theta_0$ satisfies
$$
\theta_0 \in Y \equiv L^1(\mathbb{R}^2)\cap L^\infty(\mathbb{R}^2) \cap B^{s,A}_{q,\infty}(\mathbb{R}^2)
$$
with $s>1$, $2 \le q \le \infty$ and $A$ given in (\ref{PA}).  Then (\ref{Log-Euler}) has a unique global solution $\theta$  satisfying
$$
\theta \in L^\infty\left([0,\infty); Y\right).
$$
\end{cor}

The assumption that $A_j =(j+1)^b$ with $b\le 1$ corresponds to the Besov and the sub-Besov spaces. We can also consider the solutions of (\ref{Log-Euler}) in super-Besov spaces by taking $A_j = (j+1)^b$ for $b>1$. It is easy to see that (\ref{ed}) remains valid if $s\, b <2\alpha$. Therefore (\ref{Log-Euler}) with $2\alpha >s\, b$ has a global solution in the super-Besov space $B^{s,A}_{q,\infty}$ with $A_j=(j+1)^b$ for $b>1$.

\vskip .1in
Another very important special case is when
\begin{equation}\label{ajj}
A_j =j+1, \quad
P(\xi) = |\xi|^\beta (\log(1 +|\xi|^2))^\gamma\, \quad\mbox{with $\gamma\ge 0$ and $0\le \beta < 2\alpha\le 1$}.
\end{equation}
Then again (\ref{sb}) is obviously satisfied and (\ref{decay}) is reduced to
$$
2^{s((j+1)^b - j^b)} (j+2)^{1+\gamma}\, 2^{(\beta-2\alpha)j} \to 0 \quad \mbox{as}\,\,j\to \infty,
$$
which is clearly true. That is, the following corollary holds.
\begin{cor} \label{sss}
Consider the active scalar equation
\begin{equation} \label{BG}
\left\{
\begin{array}{l}
\pp_t \theta + u\cdot\nabla \theta + \kappa (-\Delta)^\alpha\theta=0, \\
u = \nabla^\perp \psi, \quad
\Delta \psi = \Lambda^\beta\,\left(\log(1 -\Delta)\right)^\gamma\, \theta
\end{array}
\right.
\end{equation}
with $\kappa>0$, $\alpha>0$, $0\le \beta< 2\alpha\le 1$ and $\gamma\ge 0$. Assume the initial data $\theta_0\in Y\equiv L^1(\mathbb{R}^2)\cap L^\infty(\mathbb{R}^2) \cap B^{s,A}_{q,\infty}(\mathbb{R}^2)$ with $s>1$, $2\le q\le \infty$ and $A_j$ given by (\ref{ajj}). Then (\ref{BG}) has a unique global solution $\theta$ satisfying
$$
\theta \in L^\infty\left([0,\infty); Y\right).
$$
\end{cor}

Again we could have studied the global solutions of (\ref{BG}) in a super-Besov space $B^{s,A}_{q,\infty}$ with, say $A_j =(j+1)^b$ for $b>1$. Of course we need to put more restrictions on $\alpha$.  When $\gamma=0$, (\ref{BG}) becomes
\begin{equation} \label{GBG}
\left\{
\begin{array}{l}
\pp_t \theta + u\cdot\nabla \theta + \kappa (-\Delta)^\alpha\theta=0, \\
u = \nabla^\perp \psi, \quad
\Delta \psi = \Lambda^\beta\, \theta,
\end{array}
\right.
\end{equation}
which we call the generalized SQG equation. Corollary \ref{sss} does not cover the case when $\beta=2\alpha$, namely the modified SQG equation.  The global regularity of the modified SQG equation with any $L^2$ initial data has previously been obtained in \cite{CIW}. In the supercritical case when $\beta>2\alpha$, the global regularity issue for (\ref{GBG}) is open. In particular, the global issue for supercritical SQG equation ($\beta=1$ and $2\alpha<1$) remains outstandingly open.

\vskip .1in
Following the ideas in \cite{Cha} and \cite{CMT}, we approach the global issue of (\ref{GBG}) in the super case $\beta>2\alpha$ by considering the geometry of the level curves of its solution. We present a geometric type criterion for the
regularity of solutions of (\ref{GBG}). This sufficient condition controls the
regularity of solutions in terms of the space-time integrability of
$|\nao \theta| $ and the regularity of the direction field
$\xi=\nao\theta/|\nao\theta|$ (unit tangent vector to a level curve
of $\theta$).

\begin{thm} \label{crit3}
Consider (\ref{GBG}) with $\kappa > 0$, $\alpha>0$ and $0\le
\beta\le 1$. Let $\theta$ be the solution of (\ref{GBG})
corresponding to the initial data $\theta_0 \in H^m(\mathbb{R}^2)$
with $m>2$. Let $T>0$. Suppose there exists $\sigma \in (0,1)$,
$q_1\in (\frac{2}{1+\beta-\sigma} , \infty]$, $p_1\in (1, \infty]$,
$p_2 \in (1, \frac{2}{1+\sigma-\beta} )$ and $r_1, r_2 \in [1, \infty]$
such that the followings hold.
\bq\label{con220}
 \xi\in L^{r_1}(0,T; \mathcal{\dot{F}}^\sigma_{p_1,q} (\mathbb R^2)) \quad \mbox{and}
 \quad \na \in L^{r_2} (0, T; L^{p_2} (\mathbb R^2 ))\\
 \mbox{with}\qquad
\frac{1}{p_1} + \frac{1}{p_2} + \frac{\a}{r_1} +\frac{\a}{r_2}
  \leq \a+\frac12(1+\sigma-\beta) .\n
  \eq
Then $\theta$ remains in $H^m (\mathbb R^2)$ on $[0,T]$.  Especially,
when $p_1=r_1=q=\infty$, (\ref{con220}) becomes
 \bqn \xi \in
L^\infty(0,T; C^\sigma (\mathbb R^2)) \quad\mbox{and}\quad \nabla^\perp
\theta
 \in L^{r_2} (0, T; L^{p_2} (\mathbb R^2 ))\quad \\
 \mbox{with}\qquad
 \frac{1}{p_2}  +\frac{\a}{r_2}
  \leq \a+\frac12(1+\sigma-\beta).
\eqn
\end{thm}
Here $\dot{\mathcal{F}}^s_{p,q} (\mathbb{R}^2)$ denotes a homogeneous
Trebiel-Lizorkin type space. For $0\le s\le 1$, $1\le p\le \infty$ and $1\le q\le \infty$, $\dot{\mathcal{F}}^s_{p,q}$ contains functions such that the following semi-norm is finite,
$$
\|f\|_{\dot{\mathcal{F}}^s_{p,q}} = \left\{
\begin{array}{ll}
\displaystyle \left\|\left(\int \frac{|f(x+y)-f(x)|^q}{|y|^{n+sq}}\, dy\right)^{\frac1q}\right\|_{L^p}, \quad & \mbox{if $q<\infty$},\\\\
\displaystyle \left\|\sup_{y\not =0} \frac{|f(x+y)-f(x)|}{|y|^s}\right\|_{L^p}, \quad & \mbox{if $q=\infty$}
\end{array}
\right.
$$
We note that if we set $\beta=1$ in Theorem \ref{crit3}, then it reduces to Theorem 1.2 of \cite{Cha}.

\vskip .1in
The rest of this paper is divided into two sections. Section \ref{proofmain} proves Theorem \ref{main1} while Section \ref{geocri} derives the geometric regularity criterion stated in Theorem \ref{crit3}.

\vskip .4in
\section{Proof of Theorem \ref{main1}}
\label{proofmain}

This section is devoted to the proof of Theorem \ref{main1}, which involves Besov space technique and the bounds stated in Theorem \ref{nablau}. In addition, lower bound estimates associated with the fractional dissipation are also used.

\vskip .1in
\begin{proof}[Proof of Theorem \ref{main1}]
The proof is divided into two main parts. The first part establishes the global (in time) {\it a priori} bound on solutions of (\ref{general}) while the second part briefly describes the construction of a unique local (in time) solution.

\vskip .1in
For notational convenience, we write $Y= L^1(\mathbb{R}^d)\cap L^\infty(\mathbb{R}^d) \cap B^{s,A}_{q,\infty}(\mathbb{R}^d)$. The first part derives the global bound, for any $T>0$,
\begin{equation}\label{bdd}
\|\theta(\cdot,t)\|_{B^{s,A}_{q,\infty}} \le C(T, \|\theta_0\|_{Y}) \quad\mbox{for}\quad t\le T
\end{equation}
and we distinguish between two cases: $q<\infty$ and $q=\infty$. The dissipative term is handled differently in these two cases.

\vskip .1in
We start with the case when $q<\infty$. When the velocity field $u$ is divergence-free,  $\theta_0\in L^1 \cap L^\infty$ implies the corresponding solution $\theta$ of (\ref{general}) satisfies the {\it a priori} bound
\begin{equation} \label{mmm}
\|\theta(\cdot,t)\|_{L^1\cap L^\infty} \le \|\theta_0\|_{L^1\cap L^\infty},
\quad t \ge 0.
\end{equation}
When $u$ is not divergence-free, (\ref{mmm}) is assumed. The divergence-free condition is not used in the rest of the proof.

\vskip .1in
Let $j\ge -1$ be an integer. Applying $\Delta_j$ to (\ref{general}) and following a standard decomposition, we have
\begin{equation}\label{base1}
\partial_t \Delta_j \theta + \kappa (-\Delta)^\alpha \Delta_j \theta = J_1 + J_2 + J_3 +J_4 +J_5,
\end{equation}
where
\begin{eqnarray}
J_{1} &=& - \sum_{|j-k|\le 2}
[\Delta_j, S_{k-1}(u)\cdot\nabla] \Delta_k \theta, \label{j1t}\\
J_{2} &=& - \sum_{|j-k|\le 2}
(S_{k-1}(u) - S_j(u)) \cdot \nabla \Delta_j\Delta_k \theta, \label{j2t}\\
J_3   &=&  - S_j(u) \cdot\nabla \Delta_j\theta, \label{j3t}\\
J_{4} &=& - \sum_{|j-k|\le 2}
\Delta_j (\Delta_k u \cdot \nabla S_{k-1}
(\theta)), \label{j4t}\\
J_{5} &=& -\sum_{k\ge j-1}\Delta_j (\widetilde{\Delta}_k u\cdot\nabla
\Delta_k \theta) \label{j5t}
\end{eqnarray}
with $\widetilde{\Delta}_k = \Delta_{k-1} + \Delta_k + \Delta_{k+1}$. We multiply (\ref{base1}) by $\Delta_j\theta |\Delta_j \theta|^{q-2}$ and integrate in space. Integrating by parts in the term associated with $J_3$, we obtain
\begin{eqnarray*}
-\int_{\mathbb{R}^d} \left(S_j (u) \cdot\nabla \Delta_j\theta\right) \,\Delta_j\theta |\Delta_j \theta|^{q-2} \,dx &=& \frac1q \, \int_{\mathbb{R}^d} (\nabla\cdot S_j u) |\Delta_j \theta|^q \,dx\\
&=& \int_{\mathbb{R}^d} \widetilde{J_3}\, |\Delta_j \theta|^{q-1}\,dx,
\end{eqnarray*}
where $\widetilde{J_3}$ is given by
$$
\widetilde{J_3} = \frac1q (\nabla\cdot S_j u) |\Delta_j \theta|.
$$
Applying H\"{o}lder's inequality, we have
\begin{eqnarray}
&& \frac1q\,\frac{d}{dt} \|\Delta_j \theta\|^q_{L^q} + \kappa \int \Delta_j \theta |\Delta_j \theta|^{q-2}(-\Delta)^\alpha \Delta_j\theta\,dx  \label{root1}\\
&& \qquad\qquad \qquad \le \left(\|J_1\|_{L^q} + \|J_2\|_{L^q} + \|\widetilde{J_3}\|_{L^q} + \|J_4\|_{L^q} + \|J_5\|_{L^q}\right) \|\Delta_j \theta\|_{L^q}^{q-1}. \nonumber
\end{eqnarray}
For $j\ge 0$, we have the lower bound (see \cite{CMZ1} and \cite{Wu31})
\begin{equation}\label{low}
\int \Delta_j \theta |\Delta_j \theta|^{q-2}(-\Delta)^\alpha \Delta_j\theta \ge  C\, 2^{2\alpha j}\,\|\Delta_j \theta\|_{L^q}^q.
\end{equation}
For $j=-1$, this lower bound is invalid. Still we have
\begin{equation}\label{pos}
\int \Delta_j \theta |\Delta_j \theta|^{q-2}(-\Delta)^\alpha \Delta_j\theta \ge 0.
\end{equation}
Attention is paid to the case $j\ge 0$ first. Inserting (\ref{low}) in (\ref{root1}) leads to
$$
\frac{d}{dt} \|\Delta_j \theta\|_{L^q} + \kappa \, 2^{2\alpha j}\, \|\Delta_j \theta\|_{L^q} \le \|J_1\|_{L^q} + \|J_2\|_{L^q} + \|\widetilde{J_3}\|_{L^q} + \|J_4\|_{L^q} + \|J_5\|_{L^q}.
$$
By a standard commutator estimate,
$$
\|J_1\|_{L^q} \le C  \sum_{|j-k|\le 2} \|\nabla S_{k-1} u\|_{L^\infty} \|\Delta_k \theta\|_{L^q}.
$$
By H\"{o}lder's and Bernstein's inequalities,
$$
\|J_2\|_{L^q} \le C\, \|\nabla \widetilde{\Delta}_j u\|_{L^\infty} \, \|\Delta_j \theta\|_{L^q}.
$$
Clearly,
$$
\|\widetilde{J_3}\|_{L^q} \le C\,\|\nabla\cdot S_j u\|_{L^\infty} \, \|\Delta_j \theta\|_{L^q}.
$$
For $J_4$ and $J_5$, we have
\begin{eqnarray*}
\|J_4\|_{L^q} &\le&  \sum_{|j-k|\le 2} \|\Delta_k u\|_{L^\infty} \, \|\nabla S_{k-1} \theta\|_{L^q},\\
\|J_5\|_{L^q} &\le& \sum_{k\ge j-1} \,\|\widetilde{\Delta}_k u\|_{L^\infty} \| \Delta_k \nabla \theta\|_{L^q} \\
&\le& C\, \sum_{k\ge j-1} \|\nabla \widetilde{\Delta}_k u\|_{L^\infty}\, \|\Delta_k \theta\|_{L^q}.
\end{eqnarray*}
These terms can be further bounded as follows. By Theorem \ref{nablau},
\begin{eqnarray*}
\|\nabla S_k u\|_{L^\infty} & \le & \|\theta_0\|_{L^1\cap L^\infty} + C k\,P(2^{k+1})\|S_{k+1} \theta\|_{L^\infty}\\
&\le& \|\theta_0\|_{L^1\cap L^\infty} + C k\,P(2^{k+1}) \|\theta_0\|_{L^\infty}.
\end{eqnarray*}
Thus,
\begin{eqnarray*}
\|J_1\|_{L^q} &\le& C\, \|\theta_0\|_{L^1\cap L^\infty} \sum_{|j-k|\le 2} (1+ C k\,P(2^{k+1})) 2^{-s A_k} \, 2^{s A_k} \|\Delta_k \theta\|_{L^q}\\
&\le& C\, 2^{-s A_j}\,\|\theta_0\|_{L^1\cap L^\infty} \|\theta\|_{B^{s,A}_{q,\infty}}\,\sum_{|j-k|\le 2}(1+ C k\,P(2^{k+1})) 2^{s(A_j-A_k)}.
\end{eqnarray*}
Since $A_j$ is a nondecreasing function of $j$,
\begin{equation}\label{ajk}
2^{s (A_j -A_k)} \le 2^{s (A_j-A_{j-2})} \quad\mbox{for}\quad |k-j|\le 2,
\end{equation}
where we have adopted the convention that $A_l\equiv 0$ for $l<-1$. Consequently,
$$
\|J_1\|_{L^q} \le  C\, 2^{-s A_{j-2}}\,\|\theta_0\|_{L^1\cap L^\infty} \|\theta\|_{B^{s,A}_{q,\infty}}\,\left(1+ (j+2)P(2^{j+2})\right).
$$
Clearly, $\|J_2\|_{L^q}$ and $\|J_3\|_{L^q}$ admits the same bound
as $\|J_1\|_{L^q}$. By Bernstein's inequality and Theorem \ref{nablau},
\begin{eqnarray*}
\|J_4\|_{L^q} &\le& C\,\sum_{|j-k| \le 2} \|\nabla \Delta_k u\|_{L^q}\, \|S_{k-1} \theta\|_{L^\infty} \\
&\le& C\,\|\theta\|_{L^\infty} \sum_{|j-k| \le 2} P(2^{k+1}) \|\Delta_k \theta\|_{L^q}.
\end{eqnarray*}
By (\ref{ajk}), we have
$$
\|J_4\|_{L^q} \le C\, 2^{-s A_{j-2}}\,\|\theta_0\|_{L^\infty}\, \|\theta\|_{B^{s,A}_{q,\infty}}\,P(2^{j+2}).
$$
By Theorem \ref{nablau},
\begin{eqnarray*}
\|J_5\|_{L^q} &\le& C\, \sum_{k\ge j-1} P(2^{k+1}) \|\widetilde{\Delta}_k \theta\|_{L^\infty}\|\Delta_k \theta\|_{L^q}\\
 &\le& C\, \|\theta_0\|_{L^\infty} \sum_{k\ge j-1} P(2^{k+1}) \|\Delta_k \theta\|_{L^q}\\
&\le& C\, \|\theta_0\|_{L^\infty} 2^{-s A_{j-2}}\, P(2^{j+1}) \|\theta\|_{B^{s,A}_{q,\infty}}\, \sum_{k\ge j-1} \frac{2^{s A_{j-2}}}{P(2^{j+1})}\, \frac{P(2^{k+1})}{2^{s A_k}}
\end{eqnarray*}
By (\ref{sb}),
$$
\|J_5\|_{L^q}  \le C\, \|\theta_0\|_{L^\infty} 2^{-s A_{j-2}}\, P(2^{j+1}) \|\theta\|_{B^{s,A}_{q,\infty}}.
$$
Collecting all the estimates, we have, for $j\ge 0$,
\begin{eqnarray*}
\frac{d}{dt} \|\Delta_j \theta\|_{L^q} + \kappa \,2^{2\alpha j}\, \|\Delta_j \theta\|_{L^q} &\le&  C\, 2^{-s A_{j-2}}\,\|\theta_0\|_{L^1\cap L^\infty}\\
&& \,\times \|\theta\|_{B^{s,A}_{q,\infty}}\,\left(1+ (j+2)P(2^{j+2})\right).
\end{eqnarray*}
That is,
$$
\frac{d}{dt} \left(e^{\kappa 2^{2\alpha j} t} \|\Delta_j \theta\|_{L^q}\right) \le C\, e^{\kappa 2^{2\alpha j} t}2^{-s A_{j-2}}\,\|\theta_0\|_{L^1\cap L^\infty} \|\theta\|_{B^{s,A}_{q,\infty}}\,\left(1+ (j+2)P(2^{j+2})\right).
$$
Integrating in time and multiplying by $2^{s A_j}\cdot e^{ -\kappa 2^{2\alpha j} t}$, we obtain, for $j\ge 0$,
\begin{equation}\label{aj1}
2^{s A_j}\,\|\Delta_j \theta\|_{L^q}  \le 2^{s A_j}\,e^{ -\kappa 2^{2\alpha j} t}\|\Delta_j \theta_0\|_{L^q} + K_j,
\end{equation}
where
\begin{eqnarray*}
K_j = C\,\|\theta_0\|_{L^1\cap L^\infty} \,\left(1+ (j+2)P(2^{j+2})\right) 2^{s (A_j-A_{j-2})}\int_0^t e^{-\kappa 2^{2\alpha j} (t-\tau)} \|\theta(\tau)\|_{B^{s,A}_{q,\infty}}\,d\tau.
\end{eqnarray*}
To further the estimates, we fix $t_0\le T$ and let $t\le t_0$. It is easy to see that $K_j$ admits the upper bound
\begin{eqnarray*}
K_j  &\le & C\,\|\theta_0\|_{L^1\cap L^\infty} \,\left(1+ (j+2)P(2^{j+2})\right) 2^{s (A_j-A_{j-2})} \\
&& \quad \times \frac{1}{\kappa 2^{2\alpha j}} \left(1-e^{-\kappa 2^{2\alpha j} t} \right)\, \sup_{0\le \tau \le t_0} \|\theta(\tau)\|_{B^{s,A}_{q,\infty}}.
\end{eqnarray*}
According to (\ref{decay}), there exists an integer $j_0$
such that, for $j\ge j_0$,
\begin{equation}\label{bdd1}
K_j\le \frac12 \sup_{0\le \tau \le t_0} \|\theta(\tau)\|_{B^{s,A}_{q,\infty}}.
\end{equation}
For $0\le j\le j_0$,
\begin{equation}\label{bdd2}
K_j \le C\,\|\theta_0\|_{L^1\cap L^\infty} \,\left(1+ (j_0+2)P(2^{j_0+2})\right) \max_{0\le j\le j_0} 2^{s(A_j-A_{j-2})} \int_0^t \|\theta(\tau)\|_{B^{s,A}_{q,\infty}}\,d\tau.
\end{equation}
We now turn to the case when $j=-1$. By combining (\ref{base1}) and (\ref{pos}) and estimating $\|J_1\|_{L^q}$ through $\|J_5\|_{L^q}$ in an similar fashion as for the case $j\ge 0$, we obtain
\begin{equation}\label{aj2}
\|\Delta_{-1} \theta(t)\|_{L^q} \le \|\Delta_{-1} \theta(0)\|_{L^q} + C\,\|\theta_0\|_{L^1\cap L^\infty}\,\int_0^t \|\theta(\tau)\|_{B^{s,A}_{q,\infty}}\, d\tau.
\end{equation}
Putting (\ref{aj1}) and (\ref{aj2}) together, we find, for any $j\ge -1$,
\begin{equation}\label{aj3}
2^{s A_j}\,\|\Delta_j \theta\|_{L^q}  \le \|\theta_0\|_{B^{s,A}_{q,\infty}} + K_j,
\end{equation}
where $K_j$ obeys the bound (\ref{bdd1}) for $j\ge j_0$ and the bound in (\ref{bdd2}) for $-1\le j<j_0$. Applying $\sup_{j\ge -1}$ to (\ref{aj3}) and using the simple fact that
$$
\sup_{j\ge -1} K_j \le \sup_{j\ge j_0} K_j + \sup_{-1 \le j < j_0} K_j,
$$
we obtain
\begin{eqnarray*}
\|\theta(t)\|_{B^{s,A}_{q,\infty}} &\le& \|\theta_0\|_{B^{s,A}_{q,\infty}} + \frac12 \sup_{0\le \tau\le t_0} \|\theta(\tau)\|_{B^{s,A}_{q,\infty}} + C(\theta_0, j_0) \int_0^t \|\theta(\tau)\|_{B^{s,A}_{q,\infty}}\,d\tau,
\end{eqnarray*}
where
$$
C(\theta_0, j_0)  = C\,\|\theta_0\|_{L^1\cap L^\infty} \,\left(1+ (j_0+2)P(2^{j_0+2})\right) \max_{0\le j\le j_0} 2^{s(A_j-A_{j-2})}.
$$
Now taking supermum over $t\in[0,t_0]$, we obtain
$$
\sup_{0\le \tau\le t_0} \|\theta(\tau)\|_{B^{s,A}_{q,\infty}} \le 2\,\|\theta_0\|_{B^{s,A}_{q,\infty}} + C(\theta_0, j_0) \int_0^{t_0} \|\theta(\tau)\|_{B^{s,A}_{q,\infty}}\,d\tau,
$$
Gronwall's inequality then implies (\ref{bdd}) for any $t\le t_0\le T$. This finishes the case when $q<\infty$.

\vskip .1in
We now turn to the case when $q=\infty$. For $j\ge 0$, applying $\Delta_j$ yields
$$
\partial_t \Delta_j \theta + S_j u \cdot \nabla (\Delta_j \theta) + \kappa (-\Delta)^\alpha \Delta_j \theta = J_1 + J_2 + J_4 +J_5
$$
where $J_1$, $J_2$, $J_4$ and $J_5$ are as defined in (\ref{j1t}), (\ref{j2t}), (\ref{j4t}) and (\ref{j5t}), respectively. According to
Lemma \ref{localm} below, we have
\begin{equation} \label{mmmjjj}
\partial_t \|\Delta_j \theta\|_{L^\infty} + C\, 2^{2\alpha j} \|\Delta_j\theta\|_{L^\infty} \le \|J_1\|_{L^\infty} + \|J_2\|_{L^\infty} + \|J_4\|_{L^\infty} + \|J_5\|_{L^\infty}.
\end{equation}
The terms on the right can be estimated similarly as in the case when $q<\infty$. For $j=-1$, (\ref{mmmjjj}) is replaced by
$$
\partial_t \|\Delta_{-1} \theta\|_{L^\infty} \le \|J_1\|_{L^\infty} + \|J_2\|_{L^\infty} + \|J_4\|_{L^\infty} + \|J_5\|_{L^\infty}.
$$
The rest of the proof for this case is then very similar to the case $q<\infty$ and we thus omit further details.

\vskip .1in
We briefly describe the construction of a local solution of (\ref{general}) and prove its uniqueness. The solution is constructed through the method of successive approximation. More precisely, we consider a successive approximation sequence $\{\theta^{(n)}\}$ satisfying
\begin{equation}\label{succ}
\left\{
\begin{array}{l}
\theta^{(1)} = S_2 \theta_0, \\ \\
u^{(n)} = (u^{(n)}_j), \quad u^{(n)}_j = \mathcal{R}_l \Lambda^{-1} P(\Lambda) \theta^{(n)},\\ \\
\partial_t \theta^{(n+1)} + u^{(n)} \cdot\nabla \theta^{(n+1)} + \kappa (-\Delta)^\alpha\theta^{(n+1)} = 0,\\ \\
\theta^{(n+1)}(x,0) = S_{n+2} \theta_0
\end{array}
\right.
\end{equation}
and show that $\{\theta^{(n)}\}$ converges to a solution of (\ref{general}). It suffices to prove the following properties of $\{\theta^{(n)}\}$:
\begin{enumerate}
\item[i)] There exists $T_1>0$ such that $\theta^{(n)}$ is bounded uniformly in $B^{s,A}_{q,\infty}$ for any $t\in[0,T_1]$, namely
$$
\|\theta^{(n)}(\cdot,t)\|_{B^{s,A}_{q,\infty}} \le C_1 \|\theta_0\|_{Y}, \quad t\in [0,T_1],
$$
where $C_1$ is a constant independent of $n$.
\item[ii)] There exists $T_2>0$ such that $\eta^{(n+1)} = \theta^{(n+1)}- \theta^{(n)}$ is a Cauchy sequence in $B^{s-1,A}_{q,\infty}$,
$$
\|\eta^{(n)}(\cdot,t)\|_{B^{s-1,A}_{q,\infty}} \le C_2\, 2^{-n}, \quad t\in [0, T_2],
$$
where $C_2$ is independent of $n$ and depends on $T_2$ and $\|\theta_0\|_Y$ only.
\end{enumerate}
Since the essential ingredients in the proof of i) and ii) have appeared in proving the {\it a priori} bound, we omit the details. The uniqueness can be established by estimating the difference of any two solutions in $B^{s-1,A}_{q,\infty}$. A similar argument as in the proof of ii) would yield the desired result. This completes the proof of Theorem \ref{main1}.
\end{proof}

\vskip .1in
We have used the following lemma in the proof of Theorem \ref{main1}. It is obtained  in \cite{WaZh}.
\begin{lemma} \label{localm}
Let $j\ge 0$ be an integer. Let $\theta$, $u$ and $f$ be smooth functions solving the equation
$$
\partial_t \Delta_j \theta + u\cdot\nabla \Delta_j \theta + \kappa (-\Delta)^\alpha \Delta_j \theta =f,
$$
where $\kappa>0$ is a parameter. Assume that $\Delta_j \theta$ vanishes at infinity. Then, there exists a constant $C$ independent of $\theta$, $u$, $f$ and $j$ such that
$$
\partial_t \|\Delta_j \theta\|_{L^\infty} + C\, 2^{2\alpha j} \|\Delta_j \theta\|_{L^\infty} \le \|f\|_{L^\infty}.
$$
\end{lemma}

\vskip .3in
\section{Geometric regularity criterion}
\label{geocri}

\vskip .06in
In this section we prove Theorem \ref{crit3}. For this we recall the
following Serrin type of criterion, which is proved for $\beta=1$ in
\cite[Theorem 1.1]{Cha},
and obviously holds true for our case of  $\beta\in [0, 1]$.
\begin{thm}\label{crit30} Let $\th (x,t)$ be a solution of (\ref{GBG}) with
initial data $\theta_0\in H^m (\mathbb{R}^2)$ with $m>2$. Let $T>0$. If
there are indices $p,r$ with $\frac{1}{\a}<p<\infty$ and
$1<r<\infty$ respectively such that
 \bb\label{thm1}
\nao \th \in L^r (0,T; L^p (\Bbb R^2 )) \quad \mbox{ with}\quad
 \frac{1}{p} +\frac{\a}{r}\leq \a,
 \ee
 then $\theta$ remains in $H^m (\Bbb R^2)$ on $[0,T]$.
 \end{thm}
\vskip .1in
\begin{proof}[Proof of Theorem \ref{crit3}] Since the proof is
similar to that of Theorem 1.2 in \cite{Cha}, we will be brief here
mostly pointing out the essential changes.
 Let $p$ be an integer of the form $p=2^k$, where $k$ is a positive
 integer, and satisfy
 \bb\label{first}
 \frac{1}{\a} \leq p <\infty.
 \ee
We take operation of $\nabla^\bot$ on (\ref{GBG}), and take $L^2
(\mathbb R^2)$ inner product of it
 by \newline
 $\nao \th (x,t) |\nao \th (x,t)|^{p-2}$, and then
 substituting $u=-\nabla^\bot \Lambda^{-2+\beta} \theta $ into it, we have
 \bq\label{ve}
\lefteqn{\frac{1}{p} \frac{d}{dt} \|\nao \th(t)\|_{L^p} ^p +\k\int
 (\Lambda ^{2\a}\na )\cdot \na |\na |^{p-2} dx}\hspace{0.0in}\n \\
&&=\int (\nao \th \cdot \nabla )u \cdot \nao  \th  |\na |^{p-2}dx\n
\\
 && = \int \int [\nabla \th (x,t)\cdot \yh ] [\nao \th
(x+y,t)\cdot
\nabla \th (x,t )]\frac{dy}{|y|^{1+\beta}} |\na  (x,t)|^{p-2}dx \n \\
&&:=I,
 \eq
where the integral with respect to $y$ in the right hand side is in
the sense of principal value. The dissipation term can be estimated
 \bq\label{dis}
 \lefteqn{\k\int
 (\Lambda ^{2\a}\na )\cdot \na |\na |^{p-2} dx
 \geq \frac{\k}{p} \int \left|\Lambda ^{\a} |\na
 |^{\frac{p}{2}} \right|^2 dx }\hspace{.2in}\n \\
 &\geq& \frac{\k C_\a}{p} \left(\int |\na |^{\frac{p}{1-\a}} dx
 \right)^{1-\a}=\frac{\k C_\a}{p}\|\na
 \|_{L^{\frac{p}{1-\a}}} ^p,
 \eq
 where we used Lemma 2.4 of \cite{CC} and the embedding
 $L^2_{\alpha} (\mathbb R^2)\hookrightarrow L^{\frac{2}{1-\alpha}}
 (\mathbb R^2)$.
 Next, we estimate $I$  as follows.
 \bqn
 \lefteqn{I=\int\int (\xi ^ \bot (x,t)\cdot \yh ) [\xxy \cdot \xi^\bot (x,t)]|\nao \th (x+y, t)|
 \frac{dy}{|y|^{1+\beta}} |\nao \th (x,t ) |^p dx}\n \hspace{.0in}\\
 &&=\int\int (\xi ^\bot (x,t)\cdot \yh ) [\xxy -\xx ]\cdot\xi^\bot (x,t)|\nao \th (x+y, t)|
 \frac{dy}{|y|^{1+\beta}} |\nao\th (x,t ) |^p dx \n\\
&&\leq \int\int |\xxy -\xx | |\nao \th
(x+y,t)|\frac{dy}{|y|^{\frac{2+(\beta-1 +s)q}{q} +\frac{2-sq'}{q'}}}
|\nao \th (x,t) |^p dx \n\\
 &&\leq
 \int \left(\int \frac{|\xxy -\xx |^q}{|y|^{2+(\beta-1+s) q}} dy\right)^{\frac{1}{q}}
 \left( \int \frac{|\nao \th (x+y,t)|^{q'}}{ |y|^{2-s q'}} dy \right)^{\frac{1}{q'}} |\nao \th |^p dx \n \\
 &&\leq \|\xi \fn \left\|\{I_{s q'} ( |\nao \th |^{q'}) \}^{\frac{1}{q'}}
 \right\|_{L^{\tilde{p}_2}} \|\nao \th\|^p _{L^{p_3}},\n \\
\eqn
 where we used the fact $\xx \cdot\xi^\bot (x,t)=0$ in the second
 equality, and  H\"{o}lder's inequality in the second and the
 third inequalities with the exponents satisfying
 \bb\label{pcon}
 \frac{1}{p_1} +\frac{1}{\tilde{p}_2} +\frac{p}{p_3}=1, \qquad \frac{1}{q} +\frac{1}{q'} =1,
 \ee
 and $I_{a} (\cdot ) $, $0<a <2$, is the operator defined by the Riesz
  potential. We also set
  \bb\label{sigma}
   \sigma=\beta-1 +s
   \ee
  in the last inequality. After this, we apply Hardy-Littlewood-Sobolev inequality and Young's
  inequality  to estimate $I$, which is
  similar to the proof of Theorem 1.2 of \cite{Cha}, and  deduce
 \bb\label{last1}
 \frac{d}{dt} \|\nao \th(t)\|_{L^p} ^p +\frac{\k C_\a }{2} \| \nao
 \th(t)
\|_{L^{\frac{p}{1-\a}}} ^p \leq
 C\|\xi (t)\fn ^Q\|\nao \th (t)\|_{L^{p_2}}
^Q \|\nao \th (t)\|_{L^p} ^p,
 \ee
 where we set
 \bb\label{Q} Q=\frac{2\a p_1p_2}{(2\a+s )p_1 p_2 -2p_1 -2p_2},
 \ee
 which need to satisfy
 \bb\label{indices}
  \frac{1}{r_1}+\frac{1}{r_2}\leq \frac{1}{Q}.
 \ee
 We note that (\ref{indices}) is equivalent to
 $$
\frac{1}{p_1} + \frac{1}{p_2} + \frac{\a}{r_1} +\frac{\a}{r_2}
  \leq \a+\frac12(1+\sigma-\beta)
  $$
after substituting $Q$ and $s$ from (\ref{Q}) and (\ref{sigma})
 respectively
 into (\ref{indices}).
Since
$$
\int_0 ^T \|\xi (t)\fn ^Q\|\nao \th (t)\|_{L^{p_2}} ^Q dt
 \leq \left(\int_0 ^T\|\xi (t)\fn ^{r_1} dt\right)^{\frac{Q}{r_1}}
 \left(\int_0 ^T\|\nao \th (t)\|_{L^{p_2}} ^{r_2}
 dt\right)^{\frac{Q}{r_2}} <\infty
 $$
 by our hypothesis,
 The inequality (\ref{last1}) leads us to
$$\int_0 ^T \| \nao \th
\|_{L^{\frac{p}{1-\a}}} ^p dt <\infty.
$$
Now applying Theorem \ref{crit30}, we conclude the proof.
\end{proof}

\vskip .4in
\section*{Acknowledgements}
Chae's research was partially supported by NRF grant No.2006-0093854. Constantin's research was partially supported by NSF grant DMS 0804380. Wu's research was partially supported by NSF grant DMS 0907913. Wu thanks the Department of Mathematics at Sungkyunkwan University for its hospitality during his visit there, and thanks Professor Changxing Miao for discussions.

\end{document}